\documentclass[10pt,a4paper,leqno]{amsart}
\usepackage{amsfonts,amssymb,amsmath,amsthm}

\numberwithin{equation}{section}

\usepackage[normalem]{ulem}
\usepackage{color}

\newtheorem{theorem}{Theorem}[section]
\newtheorem*{ThmA}{Theorem A}
\newtheorem{proposition}[theorem]{Proposition}
\newtheorem{corollary}[theorem]{Corollary}
\newtheorem{lemma}[theorem]{Lemma}
\theoremstyle{remark}
\newtheorem{remark}[theorem]{Remark}

\newcommand{\eps}{\varepsilon}

\newcommand{\R}{{\mathbb R}}
\newcommand{\E}{{\mathcal E}}
\renewcommand{\S}{{\mathcal{S}}}
\renewcommand{\v}{{v}}

\newcommand{\beq}{\begin{equation}}
\newcommand{\eeq}{\end{equation}}
\newcommand{\beqa}{\begin{eqnarray}}
\newcommand{\eeqa}{\end{eqnarray}}
\newcommand{\beqn}{\begin{eqnarray}}
\newcommand{\eeqn}{\end{eqnarray}}

\begin{document}
\title[Asymptotic properties of ground states] {Asymptotic properties
  of ground states of scalar field equations with a vanishing
  parameter} \author{Vitaly Moroz} \address{Swansea University\\
  Department of Mathematics\\ Singleton Park\\ Swansea\\ SA2~8PP\\
  Wales, United Kingdom} \email{v.moroz@swansea.ac.uk} \author{Cyrill
  B.  Muratov} \address{Department of Mathematical Sciences\\ New
  Jersey
  Institute of  Technology\\
  University Heights\\ Newark NJ~07102\\ USA} \email{muratov@njit.edu}
\date{\today}

\thanks{The work of C. B. M. was supported, in part, by NSF via grants
  DMS-0718027, DMS-0908279 and DMS-1119724.}

\begin{abstract}
  We study the leading order behavior of positive solutions of the
  equation
  \begin{eqnarray*}
    -\Delta u +\eps u-|u|^{p-2}u+|u|^{q-2}u=0,\qquad
    x\in\R^N,
  \end{eqnarray*}
  where $N\ge 3$, $q>p>2$ and when $\eps>0$ is a small parameter. We
  give a complete characterization of all possible asymptotic regimes
  as a function of $p$, $q$ and $N$. The behavior of solutions depends
  sensitively on whether $p$ is less, equal or bigger than the
  critical Sobolev exponent $p^\ast=\frac{2N}{N-2}$. For $p<p^\ast$
  the solution asymptotically coincides with the solution of the
  equation in which the last term is absent. For $p>p^\ast$ the
  solution asymptotically coincides with the solution of the equation
  with $\eps=0$.  In the most delicate case $p=p^\ast$ the asymptotic
  behavior of the solutions is given by a particular solution of the
  critical Emden--Fowler equation, whose choice depends on $\eps$ in a
  nontrivial way.
\end{abstract}



\maketitle


\section{Introduction.}
\label{s:intro}

\subsection{Setting of the problem.}

This paper deals with the analysis of positive solutions of the scalar
field equation
$$
-\Delta u +\eps u-|u|^{p-2}u+|u|^{q-2}u=0\quad\text{in
  $\R^N$,}\leqno{(P_\eps)}
$$
where $N\ge 3$, $q>p>2$ and $\eps>0$. Specifically, we are interested
in the case where $\eps$ is a small parameter, with all other
parameters fixed.  Our goal is to understand the behavior of ground
state solutions of $(P_\eps)$ for $\eps\ll 1$.
By a {\em ground state} solution of $(P_\eps)$ we understand a
positive weak solution $u_\eps \in H^1(\R^N) \cap L^p(\R^N) \cap
L^q(\R^N)$ of $(P_\eps)$. These solutions are critical points
(saddles) of the energy
\beq \E_\eps(u):=\int_{\R^N}\frac{1}{2}|\nabla
u|^2+\frac{\eps}{2}|u|^2\,dx-\frac{1}{p}|u|^p+\frac{1}{q}|u|^q\,dx.\eeq

The existence and uniqueness of ground state solutions of $(P_\eps)$
with $\eps >0$ is well known.  The existence goes back to Strauss
\cite[Example 2]{Strauss} and Berestycki and Lions \cite[Example
2]{BL-I}. Note that by strict convexity of the integrand in
$\E_\eps(u)$ for large $|u|$ every weak solution of $(P_\eps)$ is
essentially bounded, and so by elliptic regularity these are classical
solutions of $(P_\eps)$ that decay uniformly to zero as $|x| \to
\infty$. Then the classical Gidas-Ni-Nirenberg symmetry result
\cite[Theorem 2]{GNN} implies that every ground state solution of
$(P_\eps)$ is spherically symmetric about some point.  The uniqueness
of a spherically symmetric ground state is rather delicate and was
proved only quite recently by Serrin and Tang \cite[Theorem 4
(ii)]{Serrin-Tang}.
The following theorem summarizes all the above results.

\begin{ThmA}[\cite{Strauss,BL-I,GNN,Serrin-Tang}]\label{Thm-0}
  Let $N\ge 3$ and $q>p>2$.  There exists $\eps_\ast>0$ such that
  $(P_\eps)$ has no ground state solutions for $\eps\ge\eps_\ast$,
  while for every $\eps\in(0,\eps_{\ast})$ equation $(P_\eps)$ admits
  a unique ground state solution $u_\eps \in C^\infty(\R^N)$ such that
  $u_\eps(x)$ is a monotone decreasing function of $|x|$ and there
  exists $C_\eps>0$ such that \beq \label{E-exp}
  \lim_{|x|\to\infty}|x|^{\frac{N-1}{2}}e^{\sqrt{\eps}|x|}u_\eps(x)=C_\eps>0.
  \eeq Furthermore, every ground state solution of $(P_\eps)$ is a
  translate of $u_\eps$.
\end{ThmA}

\noindent We note that the threshold value $\eps_\ast$ in Theorem
\ref{Thm-0} is simply the smallest value of $\eps > 0$ for which the
energy $\E_\eps$ is non-negative and can be easily computed
explicitly.

We are interested in the asymptotic behavior of the ground states
$u_\eps$ as $\eps\to 0$. This question naturally arises in the studies
of various bifurcation problems, for which $(P_\eps)$ can be
considered as a canonical normal form (see
e.g. \cite{cross93,vansaarloos92}). Problem $(P_\eps)$ itself may also
be considered as a prototypical example of a bifurcation problem for
elliptic equations. In fact, our results are expected to remain valid
for a broader class of scalar field equations whose nonlinearity has
the leading terms in the expansion around zero which coincide with the
ones in $(P_\eps)$. Let us also mention that problem $(P_\eps)$
appears in the studies of non-classical nucleation near spinodal in
mesoscopic models of phase transitions \cite{cahn59,Muratov,unger84},
as well as in the studies of the decay of false vacuum in quantum
field theories \cite{coleman77}.

In order to understand the asymptotic behavior of $u_\eps$ as $\eps\to
0$, we again note that for $u \geq 1$ the energy density in
$\E_\eps(u)$ is strictly convex.  Hence we may conclude that the
ground state solution $u_\eps$ in Theorem \ref{Thm-0} satisfies a
uniform upper bound \beq\label{bound} u_\eps(0)\le 1\quad\text{for all
  $\eps\in(0,\eps_\ast)$}.  \eeq Elliptic regularity then implies that
locally over compact sets the solution $u_\eps$ converges as $\eps\to
0$ to a radial solution of the limit equation
$$
-\Delta u -|u|^{p-2}u+|u|^{q-2}u=0\quad\text{in $\R^N$}.\leqno{(P_0)}
$$
It is known that (here and everywhere below $p^\ast :=
\frac{2N}{N-2}$):
\begin{itemize}
\item for $2<p\le p^\ast$ equation $(P_0)$ has no finite energy
  solutions, which is a direct consequence of Pokhozhaev's identity
  (see Remark \ref{r:nonexsup});
\item for $p>p^\ast$ equation $(P_0)$ admits a unique radial ground
  state solution.  The existence goes back to \cite[Theorem 4]{BL-I},
  see also \cite{Merle-I,Merle-II}, while the uniqueness was proved in
  \cite{Merle-II,Kwong-plus}.
\end{itemize}
Note that the natural energy space for equation $(P_\eps)$ is the
usual Sobolev space $H^1(\R^N)=\{u\in L^2(\R^N):\|\nabla
u\|_{L^2}<\infty\}$, while for $p\ge p^\ast$ the limit equation
$(P_0)$ is variationally well-posed in the homogeneous Sobolev space
$D^1(\R^N)$, defined as the completion of $C^\infty_0(\R^N)$ with
respect to the Dirichlet norm $\|\nabla u\|_{L^2}$.  Clearly,
$H^1(\R^N)\subsetneq D^1(\R^N)$ and as a consequence, no natural
perturbation setting (in the spirit of the implicit function theorem)
is available to analyze the family of equations $(P_\eps)$ as $\eps\to
0$.  In fact, a linearization of $(P_0)$ around the ground state
solution is not a Fredholm operator and has zero as the bottom of the
essential spectrum in $L^2(\R^N)$.  As a consequence, advanced
Lyapunov--Schmidt type reduction methods of Ambrosetti and Malchiodi
\cite{Ambrosetti} are not applicable to the family of equations
$(P_\eps)$.

If we introduce the {\em canonical} rescaling associated with the
lowest order nonlinear term in $(P_\eps)$: \beq\label{R-sub}
v(x)=\eps^{-\frac{1}{p-2}}\,u\Big(\frac{x}{\sqrt{\eps}}\Big), \eeq
then $(P_\eps)$ transforms into the equation
$$
-\Delta v + v
=|v|^{p-2}v-\eps^{\frac{q-p}{p-2}}\,|v|^{q-2}v\quad\text{in
  $\R^N$.}\leqno{(R_\eps)}
$$
The limit problem associated to $(R_\eps)$ as $\eps\to 0$ has the form
$$
-\Delta v + v =|v|^{p-2}v\quad\text{in $\R^N$.}\leqno{(R_0)}
$$
It is well-known that:
\begin{itemize}
\item for $p\ge p^\ast$ equation $(R_0)$ has no finite energy
  solutions, which is a direct consequence of Pokhozhaev's identity
  \cite{Pokhozhaev,BL-I});
\item for $2<p<p^\ast$ equation $(R_0)$ admits a unique radial ground
  state solution.  The existence goes back at least to \cite{Strauss},
  the uniqueness was proved in \cite{Kwong}.
\end{itemize}
The advantage of the rescaling \eqref{R-sub} is that at least in the
range $2<p\le p^\ast$ both $(R_\eps)$ and the limit problem $(R_0)$
are variationally well--posed in the same Sobolev space
$H^1(\R^N)$. Then the rescaled problem $(R_\eps)$ could be naturally
seen as a small perturbation of the limit problem $(R_0)$ and the
family of ground states $(v_\eps)$ of problem $(R_\eps)$ could be
rigorously interpreted as a perturbation of the ground state solution
of the limit problem $(R_0)$.  This could be done e.g. by using a
combination of the variational and Lyapunov--Schmidt perturbation
techniques as developed by Ambrosetti, Malchiodi et al., see
\cite{Ambrosetti} and further references therein.

The distinction between the asymptotic behaviors of the solutions of
problem $(P_\eps)$ as $\eps \to 0$ depending on the value of $p$ as
compared to $p^\ast$ was first pointed out in \cite{Muratov}. There it
was also observed that the asymptotic behavior of the ground states
$u_\eps$ for $p=p^\ast$ is not controlled by the solution set
structure of either $(P_0)$ or $(R_0)$.  Formal asymptotic analysis of
\cite{Muratov} explains that, in fact, three different asymptotic
regimes have to be distinguished in $(P_\eps)$: the {\em subcritical
  case} $2<p<p^\ast$, the {\em supercritical case} $p>p^\ast$ and the
most delicate {\em critical case} $p=p^\ast$.

It this work, using an adaptation of the constrained minimization
techniques developed by H. Berestycki and P.-L. Lions in \cite{BL-I},
combined with the Pokhozhaev identities associated with $(P_\eps)$ and
relevant limit problems, we provide a complete analysis of these three
asymptotic regimes.  The analysis confirms and extends the ideas
introduced in \cite{Muratov} and gives a full characterization of the
asymptotic behavior of ground state solutions of $(P_\eps)$ for
$\eps\to 0$.  \smallskip

\subsection*{Notations.}

For $\eps \ll 1$ and $f(\eps), g(\eps) \geq 0$, we write $f(\eps)
\lesssim g(\eps)$, $f(\eps) \sim g(\eps)$ and $f(\eps) \simeq
g(\eps)$, implying that there exists $\eps_0 > 0$ such that for every
$0 < \eps \leq \eps_0$: \smallskip

$f(\eps)\lesssim g(\eps)$ if there exists $C>0$ independent of $\eps$
such that $f(\eps) \le C g(\eps)$;

$f(\eps)\sim g(\eps)$ if $f(\eps)\lesssim g(\eps)$ and
$g(\eps)\lesssim f(\eps)$;

$f(\eps)\simeq g(\eps)$ if $f(\eps) \sim g(\eps)$ and $\lim_{\eps\to
  0}\frac{f(\eps)}{g(\eps)}=1$.

\noindent
We also use the standard notations $f = O(g)$ and $f = o(g)$, bearing
in mind that $f \geq 0$ and $g \geq 0$.
As usual, $C,c,c_1$, etc., denote generic positive constants
independent of $\eps$.

\section{Main results.}

\subsection{Subcritical case $2<p<p^\ast$.}

Since in the subcritical case the limit equation $(P_0)$ has no ground
state solutions, in view of \eqref{bound} the family of ground states
$u_\eps$ must converge to zero, locally over compact subsets of
$\R^N$.  To describe the asymptotic behavior of $u_\eps$ we use the
rescaling \eqref{R-sub} which transforms $(P_\eps)$ into equation
$(R_\eps)$.  For $2<p<p^\ast$, let $v_0(x)$ denote the unique radial
ground state solution of the limit equation $(R_0)$. It is well--known
that $v_0\in C^\infty(\R^N)$, $v_0(x)$ is a monotone decreasing
function of $|x|$ and that
\beq\lim_{|x|\to\infty}|x|^{\frac{N-1}{2}}e^{|x|}u_\eps(x)=C_0>0,\eeq
cf. \cite{BL-I}.  The advantage of the rescaling \eqref{R-sub} is that
both $(R_\eps)$ and the limit problems $(R_0)$ are variationally
well--posed in the Sobolev space $H^1(\R^N)$.  Note however that
$(R_0)$ is translationally invariant and hence the radial ground state
$v_0(x)$ is not an isolated solution.  As a consequence, an Implicit
Function Theorem argument is not directly applicable to $(R_\eps)$.
Nevertheless, it is known that the linearization operator
$-\Delta+1-(p-1)v_0^{p-2}$ of $(R_0)$ around the ground state $v_0$ is
a Fredholm operator in $H^1(\R^N)$, see \cite[Lemma 4.1]{Ambrosetti}.
Then perturbation techniques in \cite{Ambrosetti} could be easily
adapted in order to show that for all sufficiently small $\eps>0$
equation $(R_\eps)$ admits a radial ground state $v_\eps(x)$ which
converges to $v_0(x)$ as $\eps\to 0$.  Rescaling back to the original
variable and taking into account the uniqueness of the radial ground
state of $(P_\eps)$ we arrive at the following (folklore) result.

\begin{theorem}\label{th-sub}
Let $2<p<p^\ast$.
As $\eps\to 0$, the rescaled family of ground states
\beq v_\eps(x):=\eps^{-\frac{1}{p-2}}u_\eps\Big(\frac{x}{\sqrt{\eps}}\Big)\eeq
converges to $v_0(x)$ in $H^1(\R^N)$, $L^q(\R^N)$ and
$C^2(\R^N)$. In particular,
\beq u_\eps(0)\simeq \eps^{\frac{1}{p-2}}  v_0(0).\eeq
\end{theorem}

In the last section of this work we provide a short alternative proof
of this result based only upon variational methods which are developed
in the main part of this paper and without explicit references to
perturbation techniques.

\begin{remark}
  For $p\ge p^\ast$ Pokhozhaev's identity implies that $(R_0)$ has no
  nontrivial solutions in $H^1(\R^N)$.  In fact, it is known that
  $v_0(0)\to\infty$ as $p\uparrow p^\ast$.  Note that a complete
  asymptotic characterization of the ground states of equations
  $(R_0)$ as $p\uparrow p^\ast$ (and more general $m$--Laplace
  equations of type $(R_0)$) was given in
  \cite{Gazzola,Gazzola-II,Gazzola-III}.  More specifically (see
  \cite[Corollary 1]{Gazzola-II}), if $\delta:=p^\ast-p$, then for
  $\delta\downarrow 0$ it holds \beq\label{serrin} v_0(0)\simeq
  \beta_N\left\{
\begin{array}{ll}
\delta^{-\frac{N-2}{4}}, & N\ge 5,\medskip\\
\delta^{-\frac{1}{2}}|\log\delta|,& N=4,\medskip\\
\delta^{-\frac{1}{2}},& N=3,\\
\end{array}
\right.  \eeq for some explicit constants $\beta_N>0$.  This suggests
that for $p=p^\ast$ rescaling \eqref{R-sub} fails to capture the
behavior of the ground states $u_\eps$ and a different approach is
needed to handle the critical and supercritical case.  Note also that
the asymptotic behavior of ground states of ``slightly'' subcritical
elliptic problems in the context of bounded domains was studied in
\cite{Atkinson,Brezis-eps,Han,Rey}.
\end{remark}

\subsection{Supercritical case $p>p^\ast$.}
In contrast to the subcritical case, for $p>p^\ast$ the limit equation
$(P_0)$ admits a unique radial ground state solution $u_0(x)>0$.  It
is known that $u_0\in C^2(\R^N)$, $u_0(x)$ is a monotone decreasing
function of $|x|$ and that
\beq\lim_{|x|\to\infty}|x|^{N-2}u_0(x)=C_0>0,\eeq
see \cite[Theorem 4]{BL-I} or \cite{Merle-I,Merle-II} for the
existence, and \cite{Merle-II,Kwong-plus} for the uniqueness
proofs. However, as was already mentioned, the linearization operator
$-\Delta-(p-1)u_0^{p-2}$ of $(P_0)$ around the ground state $u_0$ is
not Fredholm and has zero as the bottom of the essential spectrum in
$L^2(\R^N)$. As a consequence, standard perturbation methods are not
applicable to $(P_0)$.  Using a direct analysis of the family of
constrained minimizations problem associated to $(P_\eps)$, we prove
the following.

\begin{theorem}\label{th-super}
  Let $p>p^\ast$.  As $\eps\to 0$, the family of ground states
  $u_\eps$ converges to $u_0$ in $D^1(\R^N)$, $L^q(\R^N)$ and
  $C^2(\R^N)$.  In particular, \beq u_\eps(0)\simeq u_0(0).\eeq In
  addition, $\eps\|u_\eps\|_2^2\to 0$.
\end{theorem}

\begin{remark}
  For $p=p^\ast$ Pokhozhaev's identity implies that $(P_0)$ has no
  nontrivial solutions in $D^1(\R^N)$.  In fact, it is not difficult
  to show that $u_0(0)\to 0$ as $p\downarrow p^\ast$.  Moreover, if
  $\delta:=p-p^\ast$, then for $\delta\downarrow 0$ we prove
  \beq\label{serrin-super} \delta^{\frac{1}{q-p^\ast}} \lesssim u_0(0)
  \lesssim \delta^{\frac{1}{q + N}} , \eeq and, provided that $q>
  {N(N+2) \over 2 (N - 2)}$, \beq\label{serrin-super-upper} u_0(0)
  ~\sim~ \delta^\frac{1}{q-p^\ast}. \eeq See Section \ref{S-bonus} for
  further details and full statements.  Note that related estimates
  for the asymptotics of ground states of $(P_0)$ with fixed
  $q>p>p^\ast$ on a sequence of expanding domains were studied in
  \cite{Merle-I,Merle-II}.
\end{remark}

\subsection{Critical case $p=p^\ast$.}

In the critical case both the unrescaled limit equation $(P_0)$ and
the ``canonically'' rescaled equation $(R_0)$ have no nontrivial
finite energy solutions.  We are going to show that after a suitable
rescaling the correct limit equation for $(P_\eps)$ is in fact given
by the critical Emden--Fowler equation
$$
-\Delta U = U^{p^\ast-1}\quad\text{in $\R^N$.}\leqno{(R_\ast)}
$$
It is well--known that the radial ground states of $(R_\ast)$ are
given by the function \beq\label{U-lambda-1} U_1(x):=
\left(1+\frac{|x|^2}{N(N-2)}\right)^{-\frac{N-2}{2}}, \eeq and the
family of its rescalings \beq\label{U-lambda} U_\lambda(x):=
\lambda^{-\frac{N-2}{2}}U_1\left(\frac{x}{\lambda}\right),\qquad\lambda>0.
\eeq
Our main result in this work is the following.

\begin{theorem}\label{th-main}
  Let $p=p^\ast$. There exists a rescaling
  $\lambda_\eps:(0,\eps_\ast)\to(0,\infty)$ such that as $\eps\to 0$,
  the rescaled family of ground states
\beq v_\eps(x):=\lambda_\eps^{\frac{1}{p-2}}u_\eps\Big(\lambda_\eps x\Big)\eeq
converges to $U_1(x)$ in $D^1(\R^N)$, $L^q(\R^N)$ and
$C^2(\R^N)$.  Moreover, \beq\label{lambda} \lambda_\eps\sim
\left\{
\begin{array}{ll}
  \eps^{-\frac{p-2}{2q-4}}, & N\ge 5,\medskip\\
  \big(\eps\log\frac{1}{\eps}\big)^{-\frac{1}{q-2}},& N=4,\medskip\\
  \eps^{-\frac{1}{q-4}},& N=3.\\
\end{array}
\right.
\eeq
and
\beq\label{lambda-u0}
u_\eps(0)\sim
\left\{
\begin{array}{ll}
  \eps^{\frac{1}{q-2}}, & N\ge 5,\medskip\\
  \big(\eps\log\frac{1}{\eps}\big)^{\frac{1}{q-2}},& N=4,\medskip\\
  \eps^{\frac{1}{2q-8}},& N=3.\\
\end{array}
\right.
\eeq
\end{theorem}

\begin{remark}
Asymptotics \eqref{lambda} and \eqref{lambda-u0} were first derived
in \cite{Muratov} using methods of formal asymptotic expansions.
Theorem \ref{th-main}, in particular, justifies the values of precise
asymptotic constants found in \cite{Muratov}.
\end{remark}

\subsection{Outline.}

The rest of the paper is organized as follows. In Section \ref{S-var}
we introduce a variational characterization of the ground states
$u_\eps$ of the problem $(P_\eps)$ as well as some other preliminary
results. In Section \ref{S-main} we study the critical case $p=p^\ast$
and prove Theorem \ref{th-main}.  In Section \ref{S-super} we consider
the supercritical case $p>p^\ast$ and prove Theorem \ref{th-super}.
Finally, in Section \ref{S-sub} we will revisit the subcritical case
$2<p<p^\ast$ and sketch a simple variational proof of Theorem
\ref{th-sub}, in the spirit of our previous arguments.

\section{Variational characterization of the ground states.}\label{S-var}

The existence and properties of the ground state $u_\eps$ of equation
$(P_\eps)$, as summarized in Theorem A, could be established in
several different ways, e.g. by means of ODE techniques.  Here we
shall utilize a variational characterization of the ground states
$u_\eps$ developed by Berestycki and Lions in \cite{BL-I}.

Given $q>p>2$ and $\eps\ge 0$ set \beq\label{f-eps} f_\eps(u):=\left\{
\begin{array}{cl}
0, & u<0,\\
u^{p-1}-u^{q-1}-\eps u, & u\in[0,1],\\
-\eps, & u>1,\\
\end{array}
\right.  \qquad F_\eps(u):=\int_0^u f_\eps(s)ds.  \eeq In view of
\eqref{bound} and since we are interested only in positive solutions
of $(P_\eps)$, the nonlinearity in $(P_\eps)$ may be always replaced
by its bounded truncation $f_\eps(u)$ from \eqref{f-eps}.

For $\eps>0$, consider the constrained minimization problem
$$
S_\eps:=\inf\left\{\int_{\R^N}|\nabla w|^2\,dx\Big|\; w\in
  H^1(\R^N),\;p^\ast\!\int_{\R^N}F_\eps(w)\,dx\;=1\right\}.
\leqno{(S_\eps)}
$$
As was proved in \cite[Theorem 2]{BL-I}, {there exists $\eps_\ast > 0$
  depending only on $p$ and $q$ such that for all
  $\eps\in(0,\eps_\ast)$ minimization problem $(S_\eps)$ admits a
  positive radially symmetric minimizer $w_\eps(x)$.  Further, there
  exists a Lagrange multiplier $\theta_\eps>0$ such that
  \beq\label{Euler-var} -\Delta w_\eps=\theta_\eps
  f_\eps(w_\eps)\quad\text{in $\R^N$}.  \eeq In particular, the
  minimizer $w_\eps$ satisfies Nehari's identity
  \beq\label{Nehari-var} \int_{\R^N}|\nabla w_\eps|^2\,dx=\theta_\eps
  \int f_\eps(w_\eps)w_\eps\,dx, \eeq and Pokhozhaev's identity (see
  e.g. \cite[Proposition 1]{BL-I}) \beq\label{Pokhozhaev-var}
  \int_{\R^N}|\nabla w_\eps|^2\,dx=\theta_\eps p^\ast\int
  F_\eps(w_\eps)\,dx.  \eeq The latter immediately implies that
  \beq\label{theta} \theta_\eps=S_\eps.  \eeq Then a direct
  calculation involving \eqref{theta} shows that the rescaled function
  \beq\label{var-rescale}
  u_\eps(x):=w_\eps\Big(\frac{x}{\sqrt{S_\eps}}\Big) \eeq is the
  radial ground state of $(P_\eps)$, described in Theorem A.  Another
  simple consequence of \eqref{Pokhozhaev-var} is that $(P_\eps)$ has no
  nontrivial finite energy solutions for $\eps\ge \eps_{\ast}$.

  Equivalently to $(S_\eps)$, we may seek to minimize the quotient
  \beq\S_\eps(w):=\frac{\|\nabla w\|_2^2} {\Big(p^\ast\int_{\R^N}
    F_\eps(w)\,dx\Big)^\frac{N-2}{N}},\qquad w\in\mathcal M_\eps,\eeq
  where \beq\mathcal M_\eps:=\Big\{0\le u\in D^1(\R^N),\;
  \int_{\R^N}F_\eps(w)\,dx>0 \Big\}.\eeq Clearly, if we set
  $w_\lambda(x):=w(\lambda x)$ then $\S_\eps(w_\lambda)=\S_\eps(w)$
  for all $\lambda>0$, that is $\S_\eps$ is invariant with respect to
  dilations.  This implies that \beq\label{scale-inv}
  S_\eps=\inf_{w\in{\mathcal M_\eps}}\S_\eps(w).  \eeq In addition,
  since clearly ${\mathcal M_{\eps_2}}\subset{\mathcal
    M_{\eps_1}}$ for $\eps_2>\eps_1>0$, \eqref{scale-inv} shows that
  {\em $S_\eps$ is a monotone nondecreasing function of
    $\eps\in(0,\eps_\ast)$.}

  One of the consequences of Pokhozhaev's identity \eqref{Pokhozhaev-var}
  is an expression for the total energy of the solution
  \beq\label{energy-D}
  \E_\eps(u_\eps)=\Big(\frac{1}{2}-\frac{1}{p^\ast}\Big)S_\eps^\frac{N}{2},
  \eeq see \cite[Corollary 2]{BL-I}, which shows that $u_\eps$ is
  indeed a ground state, i.e.  a nontrivial solution with the least
  energy.

  We will be frequently using the following well known decay and
  compactness properties of radial functions on $\R^N$.

\begin{lemma}
  \label{l:radial}
  \cite[Lemma A.IV, Theorem A.I$^\prime$]{BL-I}.
  \begin{enumerate}
  \item Let $s \geq 1$ and let $u \in L^s(\mathbb R^N)$ be a radial
    non-increasing function. Then for every $x \not= 0$ it holds
    \beq\label{P-2} u(x)\le C_{s,N} |x|^{-\frac{N}{s}}\|u\|_s, \eeq
    where $C_{s,N}=|B_1(0)|^{-{1 \over s}}$.
  \item Let $u_n \in H^1(\mathbb R^N)$ be a sequence of radial
    non-decreasing functions such that $u_n \rightharpoonup u$ in
    $H^1(\mathbb R^N)$. Then upon extraction of a subsequence
    \begin{align}
      \label{eq:H1rad}
      u_n \to  u ~ \text{in} ~ L^\infty(\mathbb R^N
      \backslash B_r(0)) ~\text{and}~L^s(\mathbb R^N \backslash
      B_r(0) ) \quad \forall r > 0 ~\text{and}~ \forall s > p^*.
    \end{align}
  \end{enumerate}
\end{lemma}

\section{Critical case $p=p^\ast$.}\label{S-main}

Throughout this section we always assume that $p=p^\ast$.  In this
critical case Pokhozhaev's identity implies that both the limit equation
$(P_0)$ and the canonically rescaled limit equation $(R_0)$ have no
positive finite energy solutions. We are going to show that after a
suitably chosen rescaling, the limit equation for $(P_\eps)$ is in
fact given by the critical Emden--Fowler equation.

\subsection{Critical Emden--Fowler equation.}
Let
$$
S_\ast:=\inf\left\{\int_{\R^N}|\nabla w|^2\,dx\;\Big|\;w\in
  D^1(\R^N),\;\int_{\R^N} |w|^p\,dx=1\right\}\leqno{(S_\ast)}
$$
be the optimal constant in the Sobolev inequality
\beq\int_{\R^N}|\nabla w|^2\,dx\ge S_\ast \left(\int_{\R^N}
  |w|^p\,dx\right)^{2/p}, \qquad \forall w\in D^1(\R^N).\eeq It is
easy to see that $S_\ast$ is achieved by translations of the rescaled
family \beq\label{U-tilde}
W_\lambda(x):=U_\lambda\left(\sqrt{S_\ast}x\right), \eeq where
$U_\lambda(x)$ are the ground states of the critical Emden--Fowler
equation $(R_\ast)$, explicitly defined by \eqref{U-lambda-1}.
Clearly, \beq\label{W-ast} \|{W_\lambda}\|_p=1,\qquad
\|\nabla{W_\lambda}\|_2^2=S_\ast.  \eeq A straightforward computation
leads to the explicit expression \beq\label{U-ast} \|\nabla
U_\lambda\|_2^2=\|U_\lambda\|_{p}^p=S_\ast^\frac{N}{2}.  \eeq Note
that the family of minimizers $W_\lambda$ solves the Euler--Lagrange
equation \beq -\Delta W=S_\ast W^{p-1}\quad\text{in $\R^N$}.  \eeq

\subsection{Variational estimates of $S_\eps$.}

For our purposes it is convenient to consider the dilation invariant
Sobolev quotient \beq\S_\ast(w):=\frac{\int_{\R^N}|\nabla w|^2\,dx}
{\Big(\int_{\R^N}|w|^p\,dx\Big)^\frac{N-2}{N}},\qquad w\in
D^1(\R^N),\quad w\neq 0,\eeq so that \beq S_\ast=\inf_{0\neq w\in
  D^1(\R^N)}\S_\ast(w).\eeq Denote \beq\sigma_\eps:=S_\eps-S_\ast.\eeq
In order to control $\sigma_\eps$ in terms of $\eps$, we shall use
Sobolev's minimizers $W_\lambda$ as a family of test functions for
$\S_\eps$.  Note that since $W_\lambda\in L^2(\R^N)$ only if $N\ge 5$,
we shall consider the higher and lower dimensions separately.
Straightforward calculations show that ${W_\lambda}\in L^s(\R^N)$ for
all $s>\frac{N}{N-2}$, with \beq\label{W-q}
\|{W_\lambda}\|_s^s=\lambda^{N-\frac{2s}{p-2}}\|W_1\|_s^s=\lambda^{-\frac{N-2}{2}(s-p)}\|W_1\|_{s}^s.
\eeq In particular, if $N\ge 5$ then ${W_\lambda}\in L^2(\R^N)$ and
\beq\label{W-2} \|W_\lambda\|_2^2=\lambda^2\|W_1\|_{2}^2.  \eeq To
consider dimensions $N=3,\,4$, given $R\gg \lambda$, we introduce a
cut off function $\eta_R\in C^\infty_c(\R)$ such that $\eta_R(r)=1$
for $|r|<R$, $0<\eta(r)<1$ for $R<|r|<2R$, $\eta_R(r)=0$ for $|r|>2R$
and $|\eta^\prime(r)|\le 2/R$.  We then compute as in, e.g.,
\cite[Chapter III, proof of Theorem 2.1]{Struwe}\footnote{Note that if
  $0<U\in H^1_{loc}(\R^N)$ solves \beq -\Delta U=kU^{p-1},\qquad
  x\in\R^N,\eeq for some $k\neq 0$, then \beq \int|\nabla(\eta U)|^2
  dx=k\int \eta^2|U|^p dx+\int|\nabla\eta|^2 U^2 dx\quad\text{for all
  } \eta\in C^\infty_c(\R^N).
  \eeq See also \cite[Chapter III, proof of Theorem 2.1]{Struwe}.}
\beq\label{UR-D2} \int |\nabla(\eta_R
{W_\lambda})|^2=S_\ast+O\Big(\Big(\frac{R}{\lambda}\Big)^{-(N-2)}\Big),
\eeq \beq\label{UR-p} \int
|\eta_R{W_\lambda}|^p=1-O\Big(\Big(\frac{R}{\lambda}\Big)^{-N}\Big),
\eeq \beq\label{UR-q} \int
|\eta_R{W_\lambda}|^q=\lambda^{-\frac{N-2}{2}(q-p)}\|W_1\|_{q}^q
\Big(1-O\Big(\Big(\frac{R}{\lambda}
\Big)^{-(N-2)\big(q-\frac{N}{N-2}\big)}\Big)\Big), \eeq
\beq\label{UR-2} \int
|\eta_R{W_\lambda}|^2=\lambda^{2}\|\eta_{R/\lambda}^2 W_1\|_{2}^2 =
\left\{
\begin{array}{ll}
  O(\lambda^2 \log\frac{R}{\lambda}), & N=4,\bigskip\\
  O(\lambda R), & N=3.
\end{array}
\right.
\eeq
Using the above calculations we obtain an upper estimate of
$\sigma_\eps$ which is essential for further considerations.

\begin{lemma}\label{E-S}
We have
\beq\label{sigma} 0<\sigma_\eps\lesssim \left\{
\begin{array}{ll}
  \eps^{\frac{q-p}{q-2}}, & N\ge 5,\medskip\\
  \big(\eps\log\frac{1}{\eps}\big)^\frac{q-4}{q-2},& N=4,\medskip\\
  \eps^{\frac{q-6}{2 q-8}},& N=3.\\
\end{array}
\right.
\eeq
In particular, $\sigma_\eps \to 0$ as $\eps \to 0$.
\end{lemma}

\proof To prove that $\sigma_\eps>0$ simply note that \beq
S_\ast\le\S_\ast(w_\eps)<\S_\eps(w_\eps)=S_\eps.  \eeq We shall now
establish the upper bound on $\sigma_\eps$, which clearly tends to   zero as $\eps \to 0$.

{\sc Case $N\ge 5$.}  Using ${W_\lambda}$ as a family of test
functions, we obtain that $W_\lambda \in \mathcal M_\eps$ for
sufficiently small $\eps$ and sufficiently large $\lambda$, and we
have \beq\label{S-N}
\S_\eps(W_\lambda)\le\frac{S_\ast}{\Big(1-\beta_2\eps\lambda^2-\beta_q
  \lambda^{-\frac{N-2}{2}(q-p)}\Big)^\frac{N-2}{N}}, \eeq where
\beq\beta_2:=\frac{p}{2}\|W_1\|_2^2,\qquad
\beta_q:=\frac{p}{q}\|W_1\|_q^q.\eeq
To minimize the right hand side of \eqref{S-N}, we have to minimize
the scalar function \beq\psi(\lambda):=\beta_2\eps\lambda^2+
\beta_q\lambda^{-\frac{N-2}{2}(q-p)}.\eeq It is easy to see that
$\psi$ achieves its minimum in scaling at \beq\label{Scaling-N}
\lambda_\eps = \eps^{-\frac{2}{(N-2) (q-2)}} \eeq and
\beq\min_{\lambda>0}\psi \sim \psi(\lambda_\eps) \sim
\eps^{\frac{q-p}{q-2}}.\eeq For $N\ge 5$, we conclude that
\beq\S_\eps(W_\lambda)\le
\frac{S_\ast}{\big(1-\psi(\lambda_\eps)\big)^\frac{N-2}{N}}=
S_\ast\big(1+O(\psi(\lambda_\eps)\big)
=S_\ast+O\big(\eps^{\frac{q-p}{q-2}}\big),\eeq and the bound
\eqref{sigma} is achieved on the function $W_{\lambda_\eps}$, where
$\lambda_\eps$ is given by \eqref{Scaling-N}.  \smallskip

{\sc Case $N=4$.}  Assume $R\gg\lambda$.  Testing against
$\eta_R{W_\lambda}$ and using calculations in
\eqref{UR-D2}--\eqref{UR-2} with $p = 4$, we obtain \beqn
\S_\eps(\eta_R{W_\lambda}) \le \left( S_\ast+
  O\big(\big(\frac{R}{\lambda}\big)^{-2}\big) \right) \hspace{5cm} \\
\times
{\Big(\big[1\!-\!O\big(\big(\frac{R}{\lambda}\big)^{-4}\big)\big]
  -\eps\lambda^2 O(\log \frac{R}{\lambda}) - \beta_q
  \lambda^{-(q-4)}\big[1\!-
  \!O\big(\big(\frac{R}{\lambda}\big)^{-2(q-2)}
  \big)\big]\Big)^{-\frac{1}{2}}} \nonumber \\
\le S_\ast\big(1+O(\psi(\lambda,R))\big), \nonumber \eeqn where
\beq\psi(\lambda,R) = \eps\lambda^2 O\Big(\log
\frac{R}{\lambda}\Big)+O\Big(\Big(\frac{R}{\lambda}\Big)^{-2}\Big) +
\beta_q \lambda^{-(q-4)}\big[1-o(1)\big]. \eeq Choose
\beq\label{Scaling-4}
\lambda_\eps=\Big(\eps\log\frac{1}{\eps}\Big)^{-\frac{1}{q-2}},\qquad
R_\eps=\eps^{-\frac{1}{2}}.  \eeq A routine calculation shows that as
$\eps\to 0$,
\beq\log\frac{R_\eps}{\lambda_\eps}\sim\log\frac{1}{\eps},\eeq and
hence
\beq\psi(\lambda_\eps,R_\eps)\sim\Big(\eps\log\frac{1}{\eps}\Big)^\frac{q-4}{q-2}.\eeq
Thus bound \eqref{sigma} is achieved by the test function
$\eta_{R_\eps}W_{\lambda_\eps}$, where $\lambda_\eps$ and $R_\eps$ are
given by \eqref{Scaling-4}.  \smallskip

{\sc Case $N=3$.}  Assume $R\gg\lambda$.  Testing against
$\eta_R{W_\lambda}$ and using calculations in
\eqref{UR-D2}--\eqref{UR-2} with $p = 6$, we obtain \beqn
\S_\eps(\eta_R{W_\lambda}) \le \left( {S_\ast+
    O\big(\big(\frac{R}{\lambda}\big)^{-1}\big)} \right) \hspace{5cm} \\
\times {\Big(\big[1\!-\!O\big(\big(\frac{R}{\lambda}\big)^{-3}\big)
  \big]-\eps\lambda O(R) - \beta_q
  \lambda^{-\frac{1}{2}(q-6)}\big[1\!-
  \!O\big(\big(\frac{R}{\lambda}\big)^{-(q-3)}\big)\big]
  \Big)^{-\frac{1}{3}}} \bigskip \nonumber \\
\le S_\ast\big(1+O(\psi(\lambda,R))\big), \nonumber \eeqn where
\beq\psi(\lambda,R) = \eps\lambda O\left( R \right) +
O\Big(\Big(\frac{R}{\lambda}\Big)^{-1}\Big)\Big) + \beta_q
\lambda^{-\frac{1}{2}(q-6)}\big[1-o(1)\big].\eeq Choosing
\beq\label{Scaling-3} \lambda_\eps = \eps^{-\frac{1}{q-4}},\qquad
R_\eps = \eps^{-\frac{1}{2}}, \eeq we then find that
\beq\psi(\lambda_\eps,R_\eps)\sim\eps^{\frac{q-6}{2q-8}},\eeq and the
bound \eqref{sigma} is achieved on the test function
$\eta_{R_\eps}W_{\lambda_\eps}$, where $\lambda_\eps$ and $R_\eps$ is
given by \eqref{Scaling-3}.  \qed

\subsection{Pokhozhaev estimates.}
Nehari identity \eqref{Nehari-var} combined with Pokhozhaev's identity
\eqref{Pokhozhaev-var} lead to the following important relations.

\begin{lemma}\label{L-pq}
  Set $\kappa:= \frac{q(p-2)}{2(q-p)} > 0$. Then \beq
  \|w_\eps\|_q^q=\kappa\eps\|w_\eps\|_2^2, \eeq \beq
  \|w_\eps\|_p^p=1+(\kappa+1)\eps\|w_\eps\|_2^2.  \eeq
\end{lemma}

\begin{proof}
  Since $w_\eps$ is a minimizer of $(S_\eps)$, identities
  \eqref{Nehari-var}--\eqref{theta} read as \beqn
  1&=&\|w_\eps\|_p^p-\|w_\eps\|_q^q-\eps\|w_\eps\|_2^2,\\
  1&=&\|w_\eps\|_p^p-\frac{p}{q}\|w_\eps\|_q^q-\frac{p}{2}\eps\|w_\eps\|_2^2.
  \eeqn Then the conclusion follows by a direct algebraic computation.
\end{proof}

\begin{lemma}\label{eps-2-crit}
  $\eps (\kappa+1)\|w_\eps\|_2^2\le
  \frac{N}{N-2}S_\ast^{-1}\sigma_\eps\big(1+o(1)\big).$
\end{lemma}

\begin{proof}
  Since $w_\eps$ is a minimizer of $(S_\eps)$, with the help of Lemma
  \ref{L-pq} we obtain \beq S_\ast\le\S_\ast(w_\eps)=\frac{\|\nabla
    w_\eps\|_2^2}{\|w_\eps\|_p^2}=
  \frac{S_\eps}{\Big(1+(\kappa+1)\eps\|w_\eps\|_2^2\Big)^\frac{N-2}{N}},
  \eeq or, equivalently, \beq
  S_\ast^\frac{N}{N-2}\big(1+(\kappa+1)\eps\|w_\eps\|_2^2\big)\le
  S_\eps^\frac{N}{N-2}.\eeq Since $\sigma_\eps:=S_\eps-S_\ast$,
  rearranging and differentiating, for $\eps\to 0$ we obtain \beq
  S_\ast^\frac{N}{N-2}(\kappa+1)\eps\|w_\eps\|_2^2\le
  S_\eps^\frac{N}{N-2}-S_\ast^\frac{N}{N-2}=
  \frac{N}{N-2}S_\ast^\frac{2}{N-2}\sigma_\eps+o(\sigma_\eps),\eeq so
  the conclusion follows.
\end{proof}

Combining the results of the three lemmas just proved, we obtain the
following result concerning the asymptotic behavior of different norms
associated with the minimizer $w_\eps$ of $(S_\eps)$.
\begin{corollary}
  \label{c:norms}
  As $\eps \to 0$, we have
  \begin{eqnarray}
    \label{normsasepsto0}
    \| w_\eps\|_p^p \to 1, \qquad \|w_\eps\|_q^q \to 0, \qquad \eps
    \|w_\eps\|_2^2 \to 0.
  \end{eqnarray}

\end{corollary}

\subsection{Optimal rescaling.}
Following \cite{Lions-I-1}, consider the concentration function \beq
Q_\eps(\lambda)=\int_{B_\lambda}|w_\eps|^p\,dx,\eeq where here and
everywhere below $B_\lambda$ is the ball of radius $\lambda$ centered
at the origin.  Clearly, $Q_\eps(\cdot)$ is strictly monotone
increasing, with $\lim_{\lambda \to 0} Q_\eps(\lambda)=0$ and
$\lim_{\lambda\to\infty}Q_\eps(\lambda)= \|w_\eps\|_p^p \to 1$ as
$\eps \to 0$ in view of Corollary \ref{c:norms}.
Therefore, the equation $ Q_\eps(\lambda)=Q_\ast$ with
\beq\label{Q-ast} Q_\ast:=\int_{B_1}|W_1(x)|^p\,dx < 1, \eeq has a
unique solution $\lambda = \lambda_\eps>0$ whenever $\eps \ll 1$:
\begin{eqnarray}
  \label{optimal}
  Q_\eps(\lambda_\eps) = Q_\ast.
\end{eqnarray}
Similarly, since the function \beq Q_0(\lambda) :=
\int_{B_{\lambda^{-1}}} |W_1(x)|^p \, dx = \int_{B_1} |W_\lambda(x)|^p
\, dx \eeq is strictly monotone decreasing, with $\lim_{\lambda \to 0}
Q_0(\lambda) = 1$ and $\lim_{\lambda \to \infty} Q_0(\lambda) = 0$,
there is a unique solution to the equation $Q_0(\lambda) = Q_\ast$. In
fact, by the definition of $Q_\ast$ this equation is satisfied if and
only if $\lambda = 1$.

Using the value of $\lambda_\eps$ implicitly determined by
\eqref{optimal}, we define the rescaled family \beq\label{omega}
\v_\eps(x):=\lambda_\eps^{\frac{N-2}{2}}w_{\eps}\big(\lambda_\eps
x\big).  \eeq Note that \beq \label{w-2p} \|\v_\eps\|_p=\|w_\eps\|_p=1
+ o(1),\qquad\|\nabla\v_\eps\|_2^2=\|\nabla w_\eps\|_2^2=
S_\ast+o(1),\eeq i.e. $(\v_\eps)$ is a minimizing family for $S_\ast$.
Note also that \beq\int_{B_1}|\v_\eps(x)|^p\,dx=Q_\ast.\eeq The next
statement is a direct consequence of the Concentration--Compactness
Principle of P.L.Lions, cf. \cite[Chapter I, Theorem 4.9]{Struwe}.

\begin{lemma}\label{l-dist}
  $\|\nabla\big(\v_\eps-W_1\big)\|_2\to 0$ and $\|\v_\eps-W_1\|_p\to
  0$ as $\eps\to 0$.
\end{lemma}

\proof By \eqref{w-2p}, for any sequence $\eps_n\to 0$ there exist a
subsequence $(\eps_{n'})$ such that $(\v_{\eps_{n'}})$ converges
weakly in $D^1(\R^N)$ to some radial function $w_0\in
D^1(\R^N)$. Applying the Concentration--Compactness Principle
(cf. \cite[Chapter I, Theorem 4.9]{Struwe} or \cite[Theorem
1.41]{Willem}) to $\| v_\eps \|_p^{-1} v_\eps$, we further conclude
that in fact $(\v_{\eps_{n'}})$ converges to $w_0$ strongly in
$D^1(\R^N)$ and $L^p(\R^N)$. As a consequence, $\|w_0\|_p=1$ and hence
$w_0$ is a radial minimizer of $(S_\ast)$, that is
$w_0\in\{W_\lambda\}_{\lambda>0}$.  Furthermore,
\beq\int_{B_1}|w_0(x)|^p\,dx=Q_\ast.\eeq We therefore conclude that
$w_0=W_1$.  Finally, by uniqueness of the limit the full sequence
$(\v_n)$ converges to $W_1$ strongly in $D^1(\R^N)$ and $L^p(\R^N)$.
\qed

\subsection{Rescaled equation estimates.}

The rescaled minimizer $\v_\eps$ defined in \eqref{omega} solves the
equation
$$
-\Delta v_{\eps} + S_\eps \eps\lambda_\eps^{2}\,\v_{\eps}
=S_\eps\big(|\v_{\eps}|^{p-2}\v_{\eps}-
\lambda_\eps^{-\frac{2(q-p)}{p-2}}\,|\v_{\eps}|^{q-2}\v_{\eps}\big),
\leqno{(R^{\ast}_\eps)}
$$
obtained from the Euler--Lagrange equation \eqref{Euler-var} for
$(S_\eps)$.
From the definition of $\v_\eps$ we obtain
\beq\label{w-q2} \|\v_\eps\|_q^q=\lambda_\eps^{\frac{2(q-p)}{p-2}}
\|w_\eps\|_q^q,\qquad \|\v_\eps\|_2^2
=\lambda_\eps^{-2}\|w_\eps\|_2^2.  \eeq
From Lemma \ref{L-pq} and Lemma \ref{eps-2-crit} we then derive the
essential relation \beq\label{important}
\lambda_\eps^{-\frac{2(q-p)}{p-2}}\|\v_\eps\|_q^q=
\kappa\eps\lambda_\eps^{2}\|\v_\eps\|_2^2 \lesssim \sigma_\eps, \eeq
which leads to the following two--sided estimate.

\begin{lemma}\label{lambda-main}
  $\sigma_\eps^{-\frac{p-2}{2(q-p)}} \lesssim \lambda_\eps \lesssim
  \eps^{-\frac12}\sigma_\eps^{\frac{\eps}{2}}$.
\end{lemma}
\proof Follows directly from \eqref{important} by observing that
\beq\liminf_{\eps \to 0} \|\v_\eps\|_q>0,\qquad\liminf_{\eps \to 0}
\|\v_\eps\|_2>0.\eeq To prove the latter, we note that by Lemma
  \ref{l-dist} and in view of the embedding $L^q(B_1)\subset L^p(B_1)$
  we have \beq c\|\v_\eps\chi_{B_1}\|_q\ge
\|\v_\eps\chi_{B_1}\|_p\ge\|W_1\chi_{B_1}\|_p-\|(W_1-\v_\eps)\chi_{B_1}
\|_p= \|W_1\chi_{B_1}\|_p -o(1),\eeq where here and below $\chi_{B_R}$
is the characteristic function of $B_R$. Similarly, in view of
the embedding $L^p(B_1)\subset L^2(B_1)$ we obtain
\beq\|\v_\eps\chi_{B_1}\|_2\ge\|W_1\chi_{B_1}\|_2-\|(W_1-\v_\eps)\chi_{B_1}
\|_2 = \|W_1\chi_{B_1}\|_2 -o(1),\eeq so the assertion follows.  \qed

Using estimate \eqref{sigma},
we extract from Lemma \ref{lambda-main} a lower bound
\beq\label{lambda-} \lambda_\eps ~\gtrsim~
\sigma_\eps^{-\frac{1}{2}\frac{p-2}{q-p}} ~\gtrsim~ \left\{
\begin{array}{ll}
  \eps^{-\frac{p-2}{2 q-4}}, & N\ge 5,\medskip\\
  \big(\eps\log\frac{1}{\eps}\big)^{-\frac{1}{q-2}},& N=4,\medskip\\
  \eps^{-\frac{1}{q-4}},& N=3,\\
\end{array}
\right.
\eeq
and an upper bound
\beq\label{lambda+}
\lambda_\eps ~\lesssim ~ \left\{
\begin{array}{ll}
  \eps^{-\frac{1}{2}\frac{p-2}{q-2}}, & N\ge 5,\medskip\\
  \eps^{-\frac{1}{q-2}}\big(\log\frac{1}{\eps}\big)^{\frac{q-4}{2
        q-4}},&
  N=4,\medskip\\
  \eps^{-\frac{q-2}{4 (q-4)}},& N=3.\\
\end{array}
\right.
\eeq
Note that for $N\ge 5$ the above lower and upper estimates are equivalent,
and as a consequence we obtain the following.

\begin{corollary}
Assume $N\ge 5$. Then $\|\v_\eps\|_q$ and $\|\v_\eps\|_2$ are bounded.
\end{corollary}

\proof
Follows from \eqref{important}, \eqref{lambda-} and \eqref{lambda+}.
\qed

In the lower dimensions the growth of $\|\v_\eps\|_2$ is to be taken
into account to obtain matching bounds, so instead of \eqref{lambda+}
we shall use a more explicit upper bound \beq\label{lambda+lower}
\lambda_\eps ~\lesssim~
\frac{\eps^{-1/2}\sigma_\eps^{\frac{1}{2}}}{\|\v_\eps\|_2} ~\lesssim~
\|\v_\eps\|_2^{-1}\left\{
\begin{array}{ll}
  \eps^{-\frac{1}{q-2}}\big(\log\frac{1}{\eps}\big)^{\frac{q-4}{2
        q-4}},&
  N=4,\medskip\\
  \eps^{-\frac{q-2}{4 (q-4)}},& N=3,\\
\end{array}
\right.
\eeq
which is also a combination of \eqref{important} and \eqref{sigma}.

\subsection{A lower barrier.}
To control the norm $\|\v_\eps\|_2$, we note that \beq -\Delta
\v_\eps+S_\eps\eps\lambda_\eps^2
\v_\eps=S_\eps\Big(\v_\eps^{p-1}-\lambda_\eps^{-\frac{2
      (q-p)}{p-2}}\v_\eps^{q-1}\Big)\ge -V_\eps(x)\,\v_\eps,\qquad
x\in\R^N,\eeq where \beq
V_\eps(x):=S_\eps\lambda_\eps^{-\frac{2(q-p)}{p-2}}
\v_\eps^{q-2}(x).\eeq According to the radial estimate \eqref{P-2},
\beq\label{P-2-p+} u_\eps(x)\le C_p
|x|^{-\frac{2}{p-2}}\|u_\eps\|_p, \eeq Using \eqref{w-2p} and
the   fact that $\lambda_\eps^{-\frac{2(q-p)}{p-2}} ~\lesssim~
\sigma_\eps\to 0$ by Lemmas \ref{E-S} and \ref{lambda-main}, for
sufficiently small $\eps>0$ we obtain \beq
V_\eps(x)=S_\eps\lambda_\eps^{-\frac{2(q-p)}{p-2}}\v_\eps^{q-2}(x)
\le S_\eps\lambda_\eps^{-\frac{2(q-p)}{p-2}}C_p^{q-2}
\|\v_\eps\|_p^{q-2}|x|^{-\frac{2(q-2)}{p-2}}\le
C|x|^{-\frac{2(q-2)}{p-2}},\eeq where the constant $C>0$ does not
depend on $\eps$ or $x$.  Therefore, for small $\eps>0$ solutions
$v_\eps>0$ satisfy the linear inequality \beq\label{linear}
-\Delta \v_\eps+V_0(x)\v_\eps+S_\eps\eps\lambda_\eps^2 \v_\eps\ge
0,\qquad x\in\R^N, \eeq where
$V_0(x):=C|x|^{-\frac{2(q-2)}{p-2}}$.

\begin{lemma}\label{l-barrier}
  There exists $R>0$ and $c>0$ such that for all small $\eps>0$ \beq
  v_\eps(x)\ge c|x|^{-(N-2)}e^{-\sqrt{\eps
      S_\eps}\lambda_\eps|x|}\qquad(|x|>R).\eeq
\end{lemma}

\proof Define the barrier \beq
h_\eps(x):=\big(|x|^{-(N-2)}+|x|^\beta\big)e^{-\sqrt{\eps
    S_\eps}\lambda_\eps|x|},\eeq where $\beta ~< 0$ is fixed in
such a way that \beq -(N-2)-\frac{2(q-p)}{p-2}<\beta<-(N-2),\eeq
and the value of $c>0$ will be specified later.  A direct computation
then shows that for some $R\gg 1$ one get
\begin{multline}
  -\Delta h_\eps+V_0(x)h_\eps+\eps\lambda_\eps^2 h_\eps\\
  =\big\{-\beta(\beta+N-2)|x|^{\beta-2}+C\big(|x|^{-(N-2)}
  +|x|^{\beta}\big)|x|^{-2\frac{q-2}{p-2}}\\
  +\sqrt{\eps S_\eps}\lambda_\eps\big((2\beta+N-1)|x|^{\beta-1}
  +(3-N)|x|^{-(N-1)}\big)\big\}\,e^{-\sqrt{\eps
      S_\eps}\lambda_\eps|x|}\\
  \le \big\{-\beta(\beta+N-2)|x|^{\beta-2}
  +2C|x|^{-(N-2)-2\frac{q-2}{p-2}}\big\}\,e^{-\sqrt{\eps
      S_\eps}\lambda_\eps|x|}\le 0,
\end{multline}
for all $|x|>R$, where $R\gg 1$ can be chosen independent of $\eps>0$.

Note that Lemmas \ref{l-dist} and \ref{l:radial} imply
\beq\|(\v_\eps-W_1)\chi_{B_{2R}\setminus B_{R/2}}\|_\infty\to 0,\eeq
and hence \beq\v_\eps(R)\ge \frac{1}{2}W_1(R),\eeq for all
sufficiently small $\eps>0$. Choose $c>0$ so that \beq
c(R^{-(N-2)}+R^\beta)\le \frac{1}{2}W_1(R).\eeq Then \beq\v_\eps\ge c
h_\eps\quad\text{for $|x|>R$,}\eeq by the comparison principle for the
operator $-\Delta +V_0+\eps\lambda_\eps^2$, (see, e.g.,
\cite[Theorem 2.7]{Agmon}).  \qed

\subsection{Case $N=3$ and $N=4$ completed.} We shall apply Lemma
\ref{l-barrier} to obtain matching estimates on the blow--up of
$\|\v_\eps\|_2$ in low dimensions.

\begin{lemma}
  If $N=3$ then $\|\v_\eps\|_2^2 ~\gtrsim~
  \frac{1}{\sqrt{\eps}\lambda_\eps}$.
\end{lemma}

\proof Assuming $N=3$ we directly calculate from Lemma
\ref{l-barrier}, \beq\|\v_\eps\|_2^2\ge \int_{\mathbb R^3
    \backslash B_R}|\v_\eps|^2\,dx\ge \int_R^\infty c^2
e^{-2\sqrt{\eps S_\eps}\lambda_\eps
  r}\,dr\ge\frac{C}{\sqrt{\eps}\lambda_\eps},\eeq which is what
  is required.  \qed

As an immediate corollary, using \eqref{lambda+lower}, we obtain an upper estimate of
$\lambda_\eps$ which matches the lower bound of \eqref{lambda-} in the
case $N=3$.

\begin{corollary}
  If $N=3$ then $\lambda_\eps ~\lesssim~ \eps^{-\frac{1}{q-4}}$.
\end{corollary}

Next we consider the case $N=4$.

\begin{lemma}
  If $N=4$ then $\|\v_\eps\|_2^2 ~\gtrsim~
  \log\frac{1}{\sqrt{\eps}\lambda_\eps}$.
\end{lemma}

\proof Assuming $N=4$ we directly calculate using Lemma
\ref{l-barrier}, \beq\|\v_\eps\|_2^2\ge \int_{\mathbb R^4 \backslash
  B_R}|\v_\eps|^2\,dx\ge \int_R^\infty c^2 r^{-1}e^{-2\sqrt{\eps
    S_\eps}\lambda_\eps\, r}\,dr=c^2\Gamma(0, 2\sqrt{\eps
  S_\eps}\lambda_\eps R),\eeq where \beq\Gamma(0,t)=-\log(t)-\gamma +
O(t),\qquad t\searrow 0,\eeq is the incomplete Gamma function and
$\gamma \approx 0.5772$ is the Euler constant \cite{Abramowitz}.
Hence we obtain for sufficiently small $\eps$: \beq\|\v_\eps\|_2^2\ge
c^2(-\log(2\sqrt{\eps S_\eps}\lambda_\eps R)-\gamma)\ge
C\log\Big(\frac{1}{\sqrt{\eps}\lambda_\eps}\Big), \eeq which is what
is required.  \qed

\begin{corollary}
  If $N=4$ then $\lambda_\eps ~\lesssim~
  \Big(\eps\log\frac{1}{\eps}\Big)^{-\frac{1}{q-2}}.$
\end{corollary}

\proof An immediate corollary of \eqref{important} and \eqref{sigma}
is the relation \beq C\eps\lambda_\eps^2\log\frac{1}{\sqrt{\eps}
  \lambda_\eps}\le\Big(\eps\log\frac{1}{\eps}\Big)^\frac{q-4}{q-2}.\eeq
Note that
$\eps^{\delta_1}\le\sqrt{\eps}\lambda_\eps\le\eps^{\delta_2}$ for some
$\delta_{1,2}\ge 0$ and $\eps$ small enough, which is a
consequence of \eqref{lambda+} and \eqref{lambda-}.  Therefore,
\beq\log\frac{1}{\sqrt{\eps}\lambda_\eps} ~\sim~
\log\frac{1}{\eps},\eeq and the conclusion follows.  \qed

\subsection{Further estimates.}

The results in the previous section could be used in a standard way to
improve upon some earlier estimates.

An immediate consequence of the sharp upper estimates of
$\lambda_\eps$ is the following.

\begin{corollary}\label{q-bound}
$\|\v_\eps\|_q=O(1)$.
\end{corollary}

The boundedness of the $L^q$ norm also allows to reverse
estimates of $\|\v_\eps\|_2$ via \eqref{important}.

\begin{corollary}\label{2-bound}
\beq\|\v_\eps\|_2^2=
\left\{
\begin{array}{ll}
O\big(1\big), & N\ge 5,\medskip\\
O\big(\log\frac{1}{\eps}\big),& N=4,\medskip\\
O\big(\eps^{-\frac{1}{2}\frac{q-6}{q-4}}\big),& N=3.\\
\end{array}
\right.
\eeq
\end{corollary}

We now prove that the $L^q$ bound also implies an $L^\infty$ bound.

\begin{lemma}\label{l-w}
$\|\v_\eps\|_\infty=O(1)$.
\end{lemma}

\proof Note that by $(R^\ast_\eps)$ the function $\v_\eps$ is a
positive solution of the linear inequality \beq\label{ineq} -\Delta
\v_\eps -V_\eps(x)\v_\eps\le 0,\qquad x\in\R^N, \eeq where \beq
V_\eps(x):=S_\eps \v_\eps^{p-2}(x).\eeq From the radial estimate
\eqref{P-2} we obtain \beq \label{vepsxpq} \v_\eps(x)\le
C_q\|\v_\eps\|_q|x|^{-\frac{N}{q}}.\eeq Hence, using Corollary
\ref{q-bound} we obtain \beq V_\eps(x)\le S_\eps
C_q^{p-2}\|\v_\eps\|_q^{p-2}|x|^{-\frac{N(p-2)}{q}}\le
C_\ast|x|^{-\frac{2 p}{q}}, \eeq for some constant $C_\ast>0$ which
does not depend on $\eps$ or $x$. As a consequence, $\v_\eps$ is a
positive solution of the linear inequality \beq\label{ineq-ast}
-\Delta \v_\eps -V_\ast(x)\v_\eps\le 0,\qquad x\in\R^N, \eeq where
$V_\ast(x)=C_\ast|x|^{-\frac{2p}{q}}\in L^s_{loc}(\R^N)$, for some
$s>N/2$.  The result can then be concluded by the weak Harnack
inequality for subsolutions of \eqref{ineq-ast} (cf. \cite[Remark 5.1
on p. 226]{Stampacchia}).
  Here we give an elementary proof that also works in the present
  context. Integrating the inequality in \eqref{ineq-ast} over a ball
  and applying divergence theorem, by monotonic decrease of
  $v_\eps(x)$ in $|x|$ we have
    \begin{align}
      \label{eq:gradveps}
      |\nabla v_\eps(x)| \leq {C \over |x|^{N-1}} \int_{B_{|x|}(0)}
      V_\ast(y) v_\eps(y) \, dy \leq C' v_\eps(0) |x|^{1 - {2 p \over
          q}},
    \end{align}
    for some $C, C' > 0$ independent of $\eps$ or $x$.  Integrating
    again along the straight line from $0$ to $x_0$, we obtain
    \begin{align}
      \label{eq:veps0}
      v_\eps(0) \leq v_\eps(x_0) + C'' v_\eps(0) |x_0|^{2(q-p) \over
        q},
    \end{align}
    for some $C'' > 0$ independent of $\eps$ or $x$. We then conclude
    by choosing $|x_0|$ sufficiently small independently of $\eps$,
    using \eqref{vepsxpq} and Corollary \ref{q-bound}. \qed

  A standard consequence of the $L^\infty$ bound and elliptic
  regularity theory is the following convergence statement.

  \begin{corollary}\label{cor-reg}
    $\v_\eps\to U_1$ in $C^{2}(\R^N)$ and $L^s(\R^N)$ for any $s\ge
    p$.
    In particular, \beq\label{amp} \v_\eps(0) ~\simeq~ W_1(0).  \eeq
\end{corollary}

\proof Indeed, a consequence of the $L^\infty$ bound of Lemma
\ref{l-w} and convergence in $D^1(\R^N)$ via compactness result
  for monotone radial functions in Lemma \ref{l:radial} is
convergence in $L^s(\R^N)$ for any $s\ge p$. Then Calder\'on--Zygmund
estimate \cite[Theorem 9.11]{Gilbarg} implies
convergence in $W^{2,s}_{loc}(\R^N)$ and, hence, by Sobolev
embedding also in
$C^{1,\alpha}_{loc}(\R^N)$. Since the nonlinearity in $(R_\eps^\ast)$
is smooth, using Schauder's estimates \cite[Theorem 6.2, 6.6]{Gilbarg} we
conclude convergence in $C^2_{loc}(\R^N)$.  Finally, taking into
account that the constants in Schauder estimates are uniform with
respect to translations, we deduce convergence in $C^2(\R^N)$.
\qed

Taking into account that \beq u_\eps(0) ~\sim~
\lambda_\eps^{-\frac{2}{p-2}}\v_\eps(0),\eeq we can use \eqref{amp} to
estimate the amplitude of $u_\eps(0)$ to derive \eqref{lambda-u0},
which completes the proof of Theorem \ref{th-main}.

\section{Supercritical case $p>p^\ast$.}\label{S-super}
\subsection{The limit equation.}

For $p>p^\ast$ the limit equation
$$
-\Delta u -|u|^{p-2}u+|u|^{q-2}u=0\quad\text{in $\R^N$},\leqno{(P_0)}
$$
admits a unique positive radial ground state solution $u_0\in
D^1(\R^N)$.  Further, it is known that $u_0\in C^2(\R^N)$, $u_0(x)$ is
monotone decreasing function of $|x|$, and there exists $C_0>0$ such
that \beq\label{Green-super} \lim_{|x|\to\infty}|x|^{N-2}u_0(x)=C_0>0,
\eeq see \cite[Theorem 4]{BL-I} for the existence, or
\cite{Merle-I,Merle-II} for the existence and asymptotic decay, and
\cite{Merle-II,Kwong-plus} for the uniqueness proofs.

Similarly to \eqref{var-rescale}, the ground state $u_0$ admits a
variational characterization in the Sobolev space $D^1(\R^N)$ via the
rescaling \beq\label{var-rescale-super}
u_0(x):=w_0\Big(\frac{x}{\sqrt{S_0}}\Big), \eeq where $w_0 >0$ is
the radial (i.e., depending only on $|x|$) minimizer of the
constrained minimization problem
$$
S_0:=\inf\left\{\int_{\R^N}|\nabla w|^2\,dx\Big|\; w\in
  D^1(\R^N),\;p^\ast\!\int_{\R^N}F_0(w)\,dx\;=1\right\},
\leqno{(S_0)}
$$
where $F_0$ is defined by \eqref{f-eps} (see \cite[Section 5]{BL-I}).
Similarly to \eqref{Euler-var}--\eqref{theta}, one concludes that the
minimizer $w_0$ solves the Euler--Lagrange equation \beq \label{w0S0}
-\Delta w_0=S_0\big(|w_0|^{p-2} w_0-|w_0|^{q-2} w_0 \big)\quad\text{in
  $\R^N$}. \eeq
Further, $w_0$ satisfies Nehari's identity \beq\label{Nehari-var-0}
\int_{\R^N}|\nabla
w_0|^2\,dx=S_0\int\big(|w_0|^{p}-|w_0|^{q}\big)\,dx, \eeq and
Pokhozhaev's identity (see e.g. \cite[Proposition 1]{BL-I})
\beq\label{Pokhozhaev-var-0} \int_{\R^N}|\nabla w_0|^2\,dx=S_0
p^\ast\int\left(\frac{|w_0|^{p}}{p}-\frac{|w_0|^{q}}{q}\right)\,dx.
\eeq Taking into account that $\|\nabla w_0\|_2^2=S_0$, we then derive
from Nehari and Pokho\-zhaev's identities the relation
\beq\label{Pokhozhaev-Nehari-var}
\|w_0\|_p^p-\|w_0\|_q^q=\frac{p^\ast}{p}\|w_0\|_p^p-\frac{p^\ast}{q}\|w_0\|_q^q=1,
\eeq which leads to the explicit expressions
\beq\label{pq-Pokhozhaev} \|w_0\|_p^p=\frac{(q-p^\ast)p}{(q-p)
  p^\ast},\qquad \|w_0\|_q^q=\frac{( p-p^\ast)q}{(q-p) p^\ast}.  \eeq

  \begin{remark}
    \label{r:nonexsup} Note that the arguments leading to
    \eqref{pq-Pokhozhaev} also give non-existence of non-trivial weak
    solutions $u \in D^1(\mathbb R^N) \cap L^p(\mathbb R^N) \cap
    L^q(\mathbb R^N)$ of problem $(P_0)$ in the case $2 < p \leq
    p^\ast$ and $q>p$.
  \end{remark}

\subsection{Energy and norms estimates.}

To control the relations between $S_\eps$ and $S_0$ it is convenient
to consider the equivalent to $(S_0)$ scaling invariant quotient
\beq \label{S0} \S_0(w):=\frac{\int_{\R^N}|\nabla w|^2\,dx}
{\Big(p^\ast\int_{\R^N} F_0(w)\,dx\Big)^\frac{N-2}{N}},\qquad w\in
D^1(\R^N),\quad \int_{\R^N}F_0(w)\,dx>0.\eeq Then \beq \label{S0inf}
S_0=\mathop{\inf_{w\in D^1(\R^N)}}_{F_0(w)>0}\S_0(w).\eeq

\begin{lemma}\label{L-super-2}
  $0<S_\eps-S_0\to 0$ as $\eps\to 0$.
\end{lemma}

\proof
To show that $S_0<S_\eps$ simply note that
\beq S_0\le\S_0(w_\eps)<\S_\eps(w_\eps)=S_\eps.\eeq
To control $S_\eps$ from above we will use the minimizer $w_0$ as a
test function for $(S_\eps)$.  In view of \eqref{Green-super}, we have
$w_0\in L^2(\R^N)$ if and only if $N\ge 5$.  Therefore we shall
consider the higher and lower dimensions separately.

{\sc Case $N\ge 5$.}  Testing $(S_\eps)$ against $w_0$, we obtain
\beq\label{S-N-super}
S_\eps\le\S_\eps(w_0)\le\frac{S_0}{\left(1-\eps\|w_0\|^2_{L^2(\R^N)}\right)^\frac{N-2}{N}}\le
S_0+O(\eps), \eeq which proves the claim for $N\ge 5$.  \smallskip

To consider the lower dimensions, given $R>1$ we introduce a cutoff
function $\eta_R\in C^\infty_c(\R)$ such that $\eta_R(r)=1$ for
$|r|<R$, $0<\eta(r)<1$ for $R<|r|<2R$, $\eta_R(r)=0$ for $|r|>2
R$ and $|\eta^\prime(r)|\le 2/R$.  Then taking into account
\eqref{Green-super}, for $s>\frac{N}{N-2}$ we compute
\beq\label{UR-D2-super} \int_{\R^N} |\nabla(\eta_R
w_0)|^2=S_0+O\big(R^{-(N-2)}\big), \eeq \beq\label{UR-s-super}
\int_{\R^N} |\eta_R
w_0|^s\,dx=\|w_0\|_{L^s(\R^N)}^s\Big(1-O\big(R^{N-s(N-2)}\big)\Big),
\eeq \beq\label{UR-2-super} \int_{\R^N} |\eta_R w_0|^2= \left\{
\begin{array}{ll}
O(\log(R)), & N=4,\bigskip\\
O(R), & N=3.
\end{array}
\right.
\eeq

{\sc Case $N=4$.}  Let $R=\eps^{-1}$.  Testing $(S_\eps)$ against
$\eta_R w_0$ and using the fact that $p > 4$, we obtain
\beqa\label{S-N-super-4} S_\eps\le\S_\eps(w_0)&\le&\frac{S_0+
  O\big(R^{-2}\big)}{\Big(1-O\big(R^{-4}\big)-\eps
  O\big(\log R\big)\Big)^{1/2}}\\ \nonumber
&\le&\frac{S_0+O\big(\eps^2\big)}{\Big(1-O\big(\eps^4\big)
  -O\big(\eps\log\frac{1}{\eps}\big)\Big)^\frac{1}{2}}\le
S_0+O\big(\eps\log\frac{1}{\eps}\big), \eeqa which proves the claim.

{\sc Case $N=3$.}  Let $R=\eps^{-1/2}$.  Testing $(S_\eps)$ against
$\eta_R w_0$ and using the fact that $p > 6$, we obtain
\beqa\label{S-N-super-3}
S_\eps\le\S_\eps(w_0)&\le&\frac{S_0+O\big(R^{-1}\big)}{\Big(
  1-O\big(R^{-6}\big)-\eps O\big(R\big)\Big)^{1/3}}\\
\nonumber
&\le&\frac{S_0+O\big(\eps^{1/2}\big)}{\Big(1-O\big(\eps^{3/2}
  \big)-O\big(\eps^{1/2}\big)\Big)^{1/3}}\le
S_0+O\big(\eps^{1/2}\big), \eeqa
which completes the proof.  \qed

\begin{lemma}\label{eps-s-super}
$\| w_\eps \|_{\infty} \leq 1$ and $\|w_\eps\|_s ~\lesssim~
    1$ for all $s > p^\ast$.
\end{lemma}

\begin{proof}
   In view of \eqref{bound} and \eqref{var-rescale} we have
    \beq\|w_\eps\|_\infty=\|u_\eps\|_\infty \le 1.\eeq Using Sobolev's
    inequality and Lemma \ref{L-super-2} we also obtain
    \beq\|w_\eps\|_{p^\ast}^2\le S_\ast^{-1}\|\nabla
    w_\eps\|_2^2=S_\ast^{-1}S_\eps=S_\ast^{-1}S_0\big(1+o(1)\big).\eeq
    Then for every $s>p^\ast$
    \beq\|w_\eps\|_s^s\le\|w_\eps\|_{p^\ast}^{p^\ast},\eeq so the
    assertion follows.
\end{proof}

\begin{lemma}\label{eps-2-super}
$\eps\|w_\eps\|_2^2\to 0$.
\end{lemma}

\begin{proof}
  Since $w_\eps$ is a minimizer of $(S_\eps)$, we have
  \beq\label{eps-0} 1=p^\ast\int_{\R^N}
  F_\eps(w_\eps)dx=p^\ast\int_{\R^N}
  F_0(w_\eps)dx-p^\ast\frac{\eps}{2}\|w_\eps\|_2^2.  \eeq Therefore
  \beq \S_0(w_\eps)=\frac{\|\nabla
    w_\eps\|_2^2}{\Big(p^\ast\int_{\R^N}
    F_0(w)\,dx\Big)^\frac{N-2}{N}}=
  \frac{S_\eps}{\Big(1+\frac{p^\ast}{2}\eps\|w_\eps\|_2^2\Big)^\frac{N-2}{N}}.
  \eeq Assume to the contrary of the statement of the Lemma that
  $\limsup_{\eps\to 0}\eps\|w_\eps\|_2^2=m>0$. Then by Lemma
    \ref{L-super-2} for any sequence $\eps_n\to 0$ we obtain
  \beq
  S_0\le\S_0(w_{\eps_n})=\frac{S_{\eps_n}}{\Big(1+\frac{p^\ast}{2}\eps_{n}
    \|w_{\eps_n}\|_2^2\big)\Big)^\frac{N-2}{N}}\le
  \frac{S_0\big(1+o(1)\big)}{1+\frac{p^\ast}{2}m}<S_0, \eeq a
  contradiction.
\end{proof}

\subsection{Proof of Theorem \ref{th-super}.}

Consider a sequence of $\eps_n\to 0$.  Since $\|\nabla
w_{\eps_n}\|_2^2=S_{\eps_n}\to S_0$, the sequence $(\eps_n)$ contains
a subsequence, still denoted $(\eps_n)$, such that \beq\text{
  $w_{\eps_n}\rightharpoonup \bar{w}$ in $D^1(\R^N)$ and
  $w_{\eps_n}\rightarrow \bar{w}$ a.e. in $\R^N$,} \eeq where
$\bar{w}\in D^1(\R^N)$ is a radial function.  By Lemma
\ref{eps-s-super}, the sequence $(w_{\eps_n})$ is bounded in
$L^{p^\ast}(\R^N)$ and $L^\infty(\R^N)$.  Using Lemma \ref{l:radial}
and Sobolev inequality, we also obtain a uniform bound \beq
w_\eps(x)\le C|x|^{-\frac{N-2}{2}}\|\nabla w_\eps\|_2 ~\leq 2
C|x|^{-\frac{N-2}{2}}S_0,\eeq for $\eps$ sufficiently small.  Using
Lemma \ref{l:radial} we conclude that \beq\text{$w_{\eps_n}\to
  \bar{w}$ in $L^s(\R^N)$ for any $s\in(p^\ast,\infty)$.}\eeq Taking
into account Lemma \ref{eps-2-super} and \eqref{eps-0} we also obtain
\beq\int_{\R^N}F_0(\bar{w})dx=\lim_{n \to \infty}
\int_{\R^N}F_0(w_{\eps_n})dx= \lim_{n \to
  \infty}\Big(1+p^\ast\frac{\eps_n}{2} \|w_{\eps_n} \|_2^2\Big)=1
.\eeq By the weak lower semicontinuity we also conclude that
\beq\|\nabla\bar{w}\|_2^2\le\liminf_{n \to \infty} \|\nabla
w_{\eps_n}\|_2^2=S_0,\eeq that is $\bar{w}$ is a minimizer for
$(S_0)$.  By the uniqueness of the radial minimizer of $(S_0)$ we
conclude that $\bar{w}=w_0$.

We now claim that $(w_{\eps_n})$ converges strongly to $w_0$ in
$D^1(\R^N)$.  Indeed, we have \beqa
\|\nabla(w_{\eps_n}-w_0)\|_2^2&=&\|\nabla
w_{\eps_n}\|_2^2+\|\nabla w_0\|_2^2- 2\int_{\R^N}\nabla
w_{\eps_n}\cdot\nabla w_0\,dx \\ \nonumber
&=&S_{\eps_n}+S_0-2\int_{\R^N}\nabla w_{\eps_n}\cdot\nabla w_0\,dx.
\eeqa Estimating the last term and taking into account
\eqref{w0S0}, \eqref{Pokhozhaev-Nehari-var} and the fact that by
  H\"older inequality \beqa \Big|\int_{\R^N} f_0(w_0)(w_\eps-w_0)
dx\Big|\le\|f_0(w_0)\|_{p \over p - 1} \|w_\eps-w_0\|_p \\ \le
C\|w_0\|_p^{p-1} \|w_\eps-w_0\|_p\to 0, \nonumber \eeqa we obtain
\begin{align}
  \int_{\R^N}\nabla w_{\eps_n}\cdot\nabla w_0\,dx&=S_0\int_{\R^N}
  f_0(w_0)w_{\eps_n}\,dx \\ &=S_0\int_{\R^N}
  f_0(w_0)w_0\,dx+S_0\int_{\R^N} f_0(w_0) (w_{\eps_n}-w_0) dx \nonumber \\
  \nonumber &=S_0\big(1+o(1)\big),
\end{align}
which proves the claim.

Since $(w_{\eps_n})$ converges to $w_0$ in $D^1(\R^N)$ and in
$L^s(\R^N)$ for any $s\ge p^\ast$, similarly to the proof of Corollary
\ref{cor-reg} by the standard elliptic regularity we conclude that
$(w_{\eps_n})$ converges to $w_0$ in $C^2(\R^N)$.  The proof of of
Theorem \ref{th-super} is then completed by taking into account the
uniqueness of $w_0$.

\subsection{Remarks on a slightly supercritical limit
  problem.}\label{S-bonus}

Here we discuss the asymptotic behavior as $p\downarrow p^\ast$ of the
minimizer $w_0$ of the limit variational problem $(S_0)$.  For
convenience, set $\delta:=p-p^\ast>0$. To highlight the dependance on
$\delta$, in this section we denote the ground state energy in
\eqref{S0inf} by $S_0^\delta$, while $w_0^\delta$ will be used to
denote the corresponding minimizer. Also, in this section the
  asymptotic notation such as $\lesssim$, etc., is in terms of $\delta
  \to 0$.

The following summarizes our results regarding the asymptotic behavior
of $w_0^\delta$ as $\delta\downarrow 0$.

\begin{proposition}\label{p-bonus}
  $0<S_0^\delta-S_\ast\to 0$ for $\delta\downarrow 0$. In addition, it
  holds \beq\label{serrin-super+} \delta^{\frac{1}{q-p^\ast}} \lesssim
  w_0^\delta(0) \lesssim \delta^{\frac{1}{q + N}} , \eeq and, provided
  that $q>{N (N + 2) \over 2 (N - 2)}$,
  \beq\label{serrin-super-upper+} w_0^\delta(0) \sim
  \delta^\frac{1}{q-p^\ast}. \eeq
\end{proposition}

Let us note, however, that the asymptotic of $w_0^\delta(0)$ for
general values of $q$ is open, and numerical evidence suggests that
the conclusion of \eqref{serrin-super-upper+} is {\em false} for $q$
sufficiently close to $p^\ast$.

To prove Proposition \ref{p-bonus}, we first establish a few basic
estimates for the behavior of the minimizer of the quotient in
\eqref{S0} as $\delta \to 0$.
    \begin{lemma}
      \label{l:basicSdel}
      $\| w_0^\delta \|_\infty \leq 1$, $\| \nabla w_0^\delta \|_2
      \lesssim 1$, $\| w_0^\delta \|_p \lesssim 1$ and $\| w_0^\delta
      \|_q \lesssim \delta$.
    \end{lemma}

    \begin{proof}
      The first inequality is an immediate consequence of $\| u_0
      \|_\infty \leq 1$ and \eqref{var-rescale-super}. To prove the
      second estimate, consider a suitable fixed test function $w \in
      C^\infty_0(\mathbb R^N)$ satisfying $0 \leq w \leq 1$. Then from
      \eqref{S0} and \eqref{S0inf} we can conclude that $S_0^\delta
      \lesssim 1$ as $\delta \to 0$, implying the result. The last two
      inequalities are immediate consequences of \eqref{pq-Pokhozhaev}.
    \end{proof}

    We now establish a rough upper bound on the amplitude of
    $w_0^\delta$.

\begin{lemma}
$\|w_0^\delta\|_\infty\lesssim \delta^\frac{1}{q+N}$.
\end{lemma}

\begin{proof}
  In view of the gradient estimate of Lemma \ref{l:basicSdel}, by
  Calder\'on--Zygmund inequality \cite[Theorem 9.11]{Gilbarg}
   applied to $w_0^\delta$ solving
  \eqref{w0S0} we conclude that $\| w_0^\delta \|_{W^{2,p}_{loc}(\mathbb
    R^N)}$ is uniformly bounded and, hence, $\| \nabla w_0^\delta
  \|_\infty \leq C$ for some $C > 0$ independent of $\delta$ for
  sufficiently small $\delta > 0$. This yields the following estimate
  for some $c > 0$ independent of $\delta$:
  \begin{align}
    \label{eq:vdel0inf}
    c \| w_0^\delta \|_\infty^{q+N} \leq \frac{1}{2^q} \| w_0^\delta
    \|_\infty^q |B_R(0)| \leq \int_{B_R(0)} |w_0^\delta|^q dx \leq \|
    w_0^\delta \|_q^q,
  \end{align}
  where $R = \| w_0^\delta \|_\infty / (2 C)$, and we used
  monotonicity of $w_0^\delta(x)$ in $|x|$. The result then follows
  from the fact that $\| w_0^\delta \|_q^q \sim \delta$ by
  \eqref{pq-Pokhozhaev}.
\end{proof}

The relations in \eqref{pq-Pokhozhaev} immediately lead to the
  following lower bound on $w_0^\delta(0)$.

\begin{lemma}\label{delta-lower}
$\|w_0^\delta\|_\infty\gtrsim\delta^\frac{1}{q-p^\ast}$.
\end{lemma}

\begin{proof}
  Indeed, by
  \eqref{pq-Pokhozhaev} we have \beq \delta \| w_0^\delta \|^p_p \leq {p \over
    q} (q - p^\ast) \| w_0^\delta \|_\infty^{q-p^* - \delta} \|
  w_0^\delta \|_p^p
  \eeq
    and the result follows from $\|w_0^\delta\|_p^p > 0$ and smallness
  of $\delta$.
\end{proof}

Importantly, for sufficiently large $q$ we can prove a matching
upper bound, yielding the precise asymptotic behavior of the
  minimizer's amplitude as $\delta \to 0$.

\begin{lemma}
  \label{l:qlarge}
  If $q> {N (N + 2) \over 2 (N - 2)}$ then
  $\|w_0^\delta\|_\infty\lesssim\delta^\frac{1}{q-p^\ast}$.
\end{lemma}

\begin{proof}
  In view of Lemmas \ref{l:radial} and \ref{l:basicSdel} and Sobolev
  inequality, we have \beqa\label{delta-upper} w_0^\delta(x)&\le
  &\min\big\{C_{p^\ast}|x|^{-\frac{N}{p^\ast}}
  \|w_0^\delta\|_{p^\ast},C_{q}
  |x|^{-\frac{N}{q}}\|w_0^\delta\|_{q}\big\}, \\ \nonumber &\lesssim
  &\min\big\{ |x|^{-\frac{N-2}{2}},\delta^{\frac{1}{q}}
  |x|^{-\frac{N}{q}}\big\}. \eeqa In view of \eqref{Green-super},
  \eqref{var-rescale-super} and Lemma \ref{eps-s-super}, we can apply
  Newtonian kernel to $(P_0^\delta)$. We obtain \beq
  w_0^\delta(x)=S_0^\delta A_N\int_{\R^N}\frac{(w_0^\delta(y))^{p-1}-
    (w_0^\delta(y))^{q-1}}{|x-y|^{N-2}}dy,\eeq where
  $A_N=\tfrac{\Gamma((N - 2)/2)}{4\pi^{N/2}}$.  In particular, for
  $q>\frac{N(N+2)}{2 (N - 2)}$ we have (with a slight abuse of
  notation) \beqn w_0^\delta(0)&=&{S_0^\delta \over N - 2}
  \int_0^\infty\big( (w_0^\delta(r))^{p-1}
  - (w_0^\delta(r))^{q-1}\big) r\,dr\\
  \nonumber
  &\le& S_\ast (1 + o(1)) \int_0^\infty (w_0^\delta(r))^{p^\ast-1} r\,dr\\
  \nonumber &\lesssim&\int_0^\infty \min\big\{
  r^{-\frac{N-2}{2}},\delta^{\frac{1}{q}}
  r^{-\frac{N}{q}}\big\}^{\frac{N+2}{N-2}} r\,dr,\\
  \nonumber &\lesssim& \int_R^\infty r^{-\frac{N}{2}}\,dr+
  \delta^{\frac{N+2}{q(N-2)}}\int_0^R
  r^{1-\frac{N(N+2)}{q(N-2)}}\,dr,\\ \nonumber &\lesssim&
  R^{-\frac{N-2}{2}}+
  \delta^{\frac{N+2}{q(N-2)}}R^{2-\frac{N(N+2)}{q(N-2)}}.\\
  \nonumber \eeqn Minimizing for $q>\frac{N(N+2)}{2 (N - 2)}$ the
  function \beq\psi_\delta(R):= R^{-\frac{N-2}{2}}+
  \delta^{\frac{N+2}{q(N-2)}}R^{2-\frac{N(N+2)}{q(N-2)}}\eeq we obtain
  \beq\min_{R>0}\psi_\delta(R)=\psi_\delta(R_\ast)\sim
  \delta^{\frac{1}{q-p^\ast}},\eeq where
  $R_\ast\sim\delta^{-\frac{2}{(N-2) (q- p^\ast)}}$.
\end{proof}

We also establish the energy convergence estimate.

\begin{lemma}\label{L-super-2-delta}
$0<S_0^\delta-S_\ast\to 0$ as $\delta\to 0$.
\end{lemma}

\proof Taking into account \eqref{pq-Pokhozhaev} we obtain \beq
S_\ast\le\S_\ast(w_0^\delta)= \frac{\|\nabla
  w_0^\delta\|_2^2}{\|w_0^\delta\|_p^2}= \Big(\frac{p^\ast
    (q-p)}{p(q-p^\ast)}\Big)^{2/p} S_0^\delta<S_0^\delta.  \eeq To
control $S_0^\delta$ from above we will use the Sobolev minimizers
$(W_\lambda)_{\lambda>0}$ as a family of test function for
$(S_0^\delta)$.  Using \eqref{W-ast} we obtain
\beq\label{S-N-super-delta} \S_0^\delta(W_\lambda)=
\frac{S_\ast}{\left(\frac{p^\ast}{p} \lambda^{-\frac{N-2}{2}\delta}
    \|W_1\|_{p}^{p}-\frac{p^\ast}{q}
    \lambda^{-\frac{N-2}{2}(q-p^\ast)}
    \|W_1\|_{q}^q\right)^\frac{N-2}{N}}.  \eeq
To minimize the right hand side of \eqref{S-N-super-delta}, we need to
maximize for $\lambda>0$ the scalar function
\beq\psi(\lambda):=\frac{p^*}{p}\|W_1\|_p^p
\lambda^{-\frac{1}{2}(N-2)(p-p^\ast)}-\frac{p^*}{q}\|W_1\|_q^q\,
\lambda^{-\frac{1}{2}(N-2 ) (q-p^\ast)}\,.\eeq It is easy to see
that $\psi$ achieves its maximum at
\beq\lambda_\ast:=\left(\frac{p (q-p^\ast)
      \|W_1\|_q^q}{q(p-p^\ast) \|W_1\|_p^p}
  \right)^{\frac{2}{N-2}\frac{1}{q-p}},\eeq and
\beq\psi(\lambda_\ast)=A(p,p^\ast,q)
\|W_1\|_p^{\frac{p(q-p^\ast)}{q-p}}
\|W_1\|_q^{-\frac{q(p-p^\ast)}{q-p}},\eeq where \beq
A(p,p^\ast,q):=p^\ast
p^{-\frac{q-p^\ast}{q-p}}q^{\frac{p-p^\ast}{q-p}}
\left\{\left(\frac{q-p^\ast}{p-p^\ast}\right)^{-\frac{p-p^\ast}{q-p}}-
  \left(\frac{q-p^\ast}{p-p^\ast}\right)^{-\frac{q-p^\ast}{q-p}}\right\}>0. \eeq
In particular, when $\delta=p-p^\ast\to 0$ we have $A(p, p^\ast,
  q) \simeq 1$ and $\psi(\lambda_\ast) \simeq 1$, so \beq
S_0^\delta\le\S_0^\delta(W_{\lambda_\ast})=\big(\psi(\lambda_\ast)
\big)^{-\frac{N-2}{N}} S_\ast=S_\ast\big(1+o(1)\big), \eeq which
completes the proof.  \qed

\begin{remark}
  Instead of $W_\lambda$ we can use rescalings of an arbitrary
  function $w\in D^1(\R^N)\cap L^q(\R^N)$ as a family of test function
  in \eqref{S-N-super-delta}.  Then, taking into account that by
  Sobolev imbedding $w \in L^{p^\ast}(\mathbb R^N)$ and, hence, by
  interpolation we have $w \in L^p(\mathbb R^N)$ as well, the above
  argument with generic $q>p>p^*$ leads to \beq\label{GN} S_0(p,
  p^\ast, q) A^{2 \over p^\ast}({p,p^\ast,q})\|w\|_p^{\frac{2 p
      (q-p^\ast)}{p^\ast(q-p)}} \leq \|\nabla w\|_2^2 \;
  \|w\|_q^{\frac{2 q(p-p^\ast)}{p^\ast(q-p)} }, \eeq which could be
  interpreted as a {\sl supercritical} Gagliardo--Nirenberg type
  inequality. Similar ideas where used in \cite{DelPino} to establish
  sharp constants in the classical Gagliardo--Nirenberg inequality,
  which {\em formally} coincides with \eqref{GN} when $1<q<p<p^\ast$.
\end{remark}

\section{Subcritical case $2<p<p^\ast$ revisited: proof of Theorem
  \ref{th-sub}.}\label{S-sub}

In the subcritical case Pokhozhaev's identity implies that the limit
equation $(P_0)$ has no positive finite energy solutions. As discussed
in the Introduction, to understand the asymptotic behavior of the
ground states $u_\eps$ we consider the rescaling in \eqref{R-sub},
  which transforms $(P_\eps)$ into $(R_\eps)$, with the associated
  limit problem as $\eps\to 0$ given by $(R_0)$ (see
  Sec. \ref{s:intro}).

Let $G_\eps:\R\to\R$ be a bounded $C^2$--function such that \beq
G_\eps(w):=\frac{1}{p}|w|^p-\frac{1}{2}|w|^2-
\frac{\eps^{\frac{q-p}{p-2}}}{q}|w|^q \eeq for $0\le w\le
\eps^{-\frac{1}{p-2}}$, $G_\eps(w)\le 0$ for
$w>\eps^{-\frac{1}{p-2}}$, and $G_\eps(w)=0$ for $w\le 0$.  For
$\eps\in[0,\eps^\ast)$, consider a family of the constrained
minimization problems
$$
S_\eps' :=\inf\left\{\int_{\R^N}|\nabla w|^2\,dx\Big|\; w\in
  H^1(\R^N),\;p^\ast\int_{\R^N}G_\eps(w)\,dx\;=1\right\}.
\leqno{(S_\eps')}
$$
Note that all the problems $(S_\eps')$, including the limit
problem $(S_0')$, are well posed in the same energy space
$H^1(\R^N)$.  According to \cite[Theorem 2]{BL-I}, $(S_\eps)$
admits a radial positive minimizer $w_\eps$ for every
$\eps\in[0,\eps^\ast)$.
In view of its uniqueness \cite{Kwong}, the rescaled function
\beq\label{var-rescale-super-eps}
v_\eps(x):=w_\eps\Big(\frac{x}{\sqrt{S_\eps'}}\Big), \eeq coincides
with the radial ground state of $(R_\eps)$.

In order to estimate $S_\eps'$, consider the associated dilation
invariant representation \beq\S_\eps'(w):=\frac{\int_{\R^N}|\nabla
  w|^2\,dx} {\Big(p^\ast\int_{\R^N}
  G_\eps(w)\,dx\Big)^\frac{N-2}{N}},\qquad w\in \mathcal M_\eps',\eeq
where $\mathcal M_\eps' :=\big\{0\le u\in
H^1(\R^N),\;\int_{\R^N}G_\eps(w)\,dx >0 \big\}$.  Clearly \beq S_\eps'
=\inf_{w\in\mathcal M_\eps'} \S_\eps'(w)\eeq and for sufficiently
small $\eps$ we have \beq \label{S0pSep} S_0'\le \S_0'(w_\eps)<
\S_\eps'(w_\eps)= S_\eps'.  \eeq Indeed, since by definition
  $p^\ast \int_{\mathbb R^N} G_\eps(w_\eps) dx = 1$ and $G_\eps(s)$ is
  a decreasing function of $\eps$ for each $s > 0$, we have $w_\eps
  \in \mathcal M_0'$, and the second inequality again follows by
  monotonicity of $G_\eps(s)$ in $\eps$.  At the same time, by
  continuity $w_0 \in \mathcal M_\eps'$ for sufficiently small
  $\eps$. Therefore, using $w_0$ as a test function for $(S_\eps')$,
we obtain for sufficiently small $\eps$ \beq\label{S-eps-prime}
S_\eps'\le\S_\eps'(w_0)~=~ \frac{S_0'}{\left(1-\frac{p^*}{q} \eps^{q -
      p \over p - 2} \|w_0\|_q^q\right)^\frac{N-2}{N}}\le S_0'+O\left(
  \eps^{q - p \over p - 2} \right).  \eeq Therefore, $S_\eps'\to S_0'$.

Arguing as in the proof of Lemma \ref{L-pq}, we may conclude that
  \begin{align}
    \label{eq:wepsL2Lp}
    \| w_\eps \|_p^p = {(q - p^\ast) p \over (q - p) p^\ast} + {p (q -
      2) \over 2 (q - p)} \| w_\eps \|_2^2.
  \end{align}
  Then, using this identity to compute $\mathcal S_0'(w_\eps)$ and the
  convergence of $S_\eps'$ to $S_0'$, after some tedious algebra we
  obtain
  \begin{align}
    \label{eq:S0prime}
   \lim_{\eps \to 0} \| w_\eps \|_2^2 = {2 (p^\ast - p) \over p^\ast
     (p - 2)}, \qquad \lim_{\eps \to 0} \| w_\eps \|_p^p = {(p^\ast -
     2) p \over (p - 2) p^\ast}.
  \end{align}
  In particular, this implies that $p^\ast \int_{\mathbb R^N}
  G_0(w_\eps) dx \to 1$ as $\eps \to 0$. Hence, there exists a
  rescaling $\lambda_\eps \to 1$ such that $p^\ast \int_{\mathbb R^N}
  G_0(\tilde w_\eps) dx = 1$ and $\mathcal S_\eps'(\tilde w_\eps) \to
  S_0'$ for $\tilde w_\eps(x) := w_\eps(\lambda_\eps x)$. This implies
  that $(\tilde w_\eps)$ is a minimizing family for $(S_0')$ that
  satisfies the constraint used in the analysis of \cite{BL-I}.  Then,
  applying \cite[Theorem 2]{BL-I} we conclude that for a sequence
  $\eps_n \to 0$ we have $\tilde w_{\eps_n} \to \bar w$ strongly in
  $H^1(\mathbb R^N)$, and in view of the convergence of
  $(\lambda_\eps)$ we have $w_{\eps_n} \to \bar w$ as well, where
  $\bar w$ is the minimizer of $(S_0')$ satisfying the
  constraint. Therefore, by uniqueness of minimizers of $(R_0)$
  \cite{Kwong}, we have $\bar w = w_0$ and the limit is a full limit.


  Finally, arguing as in the proof of Lemma \ref{l-w}, using $\|
  w_\eps \|_{p^\ast}$ instead of the $L^q$ norm to control the growth
  of $w_\eps$ at the origin, we also conclude that $\| w_\eps
  \|_\infty \lesssim 1$ as $\eps \to 0$. Then by standard elliptic
  regularity, similarly to the proof of Corollary \ref{cor-reg}, we
  conclude that $w_\eps$ converges to $w_0$ in $L^s(\mathbb R^N)$ for
  any $s\ge 2$ and in $C^2(\R^N)$, which completes the proof of Theorem
  \ref{th-sub}.



\end{document}